\documentclass[a4paper,11pt]{amsart}
\usepackage{amsmath,amssymb,amsfonts,latexsym}

\newtheorem{theorem}{Theorem}[section]
\newtheorem{lemma}[theorem]{Lemma}
\newtheorem{corollary}[theorem]{Corollary}

\input{tcilatex}

\begin{document}
\title[\textbf{Spectral properties of a Sturm-Liouville problem}]{\textbf{%
Asymptotic properties of eigenvalues and eigenfunctions of a Sturm-Liouville
problem with discontinuous weight function}}
\author{\textbf{Erdo\u{g}an \c{S}en}}
\address{\textbf{Department of Mathematics, Faculty of Arts and Science,
Namik Kemal University, 59030, Tekirda\u{g}, Turkey and Department of
Mathematics Engineering, Istanbul Technical University, Maslak, 34469
Istanbul, Turkey}}
\email{\textbf{erdogan.math@gmail.com}}

\begin{abstract}
In this paper, by using the similar methods of \ [O. Sh. Mukhtarov and M.
Kadakal, Some spectral properties of one Sturm-Liouville type problem with
discontinuous weight, \textit{Siberian Mathematical Journal}, 46 (2005)
681-694] we extend some spectral properties of regular Sturm-Liouville
problems to those which consist of a Sturm-Liouville equation with
discontinuous weight at two interior points together with spectral
parameter-dependent boundary conditions. We give an operator-theoretic
formulation for the considered problem and obtain asymptotic formulas for
the eigenvalues and eigenfunctions.

\vspace{2mm}\noindent \textsc{2010 Mathematics Subject Classification.}
34L20, 35R10.

\vspace{2mm}

\end{abstract}

\thanks{}
\maketitle




\section{\textbf{Introduction}}


Sturmian theory is one of the most extensively developing fields in
theoretical and applied mathematics \ The literature is voluminous and we
refer to [1-14]. The theory of discontinuous Sturm-Liouville type problems
mainly has been developed by Mukhtarov and his students (see [1-12]).
Particularly, there has been an increasing interest in the spectral analysis
of boundary-value problems with eigenvalue-dependent boundary conditions
[1-12,15-18,21,23,24].

In this paper we consider the boundary value problem for the differential
equation%
\begin{equation}
\tau u:=-u^{\prime \prime }+q(x)u=\lambda \omega (x)u  \tag{1.1}
\end{equation}%
for $x\in \left[ -1,h_{1}\right) \cup \left( h_{1},h_{2}\right) \cup \left(
h_{2},1\right] $ (i.e., $x$ belongs to $\left[ -1,1\right] $ but the two
inner points $x=h_{1}$ and $x=h_{2}$), where $q(x)$ is a real valued
function, continuous in $\left[ -1,h_{1}\right) $, $\left(
h_{1},h_{2}\right) $ and $\left( h_{2},1\right] $ with the finite limits $%
q\left( \pm h_{1}\right) =\lim_{x\rightarrow \pm h_{1}}$, $q\left( \pm
h_{2}\right) =\lim_{x\rightarrow \pm h_{2}}$; $\omega \left( x\right) $ is a
discontinuous weight function such that $\omega \left( x\right) =\omega
_{1}^{2}$ for $x\in \left[ -1,h_{1}\right) $, $\omega \left( x\right)
=\omega _{2}^{2}$ for $x\in \left( h_{1},h_{2}\right) $ and $\omega \left(
x\right) =\omega _{3}^{2}$ for $x\in \left( h_{2},1\right] $, $\omega >0$
together with the standart boundary condition at $x=-1$%
\begin{equation}
L_{1}u:=\cos \alpha u\left( -1\right) +\sin \alpha u^{\prime }\left(
-1\right) =0,  \tag{1.2}  \label{equation 2}
\end{equation}%
the spectral parameter dependent boundary condition at $x=1$%
\begin{equation}
L_{2}u:=\lambda \left( \beta _{1}^{\prime }u\left( 1\right) -\beta
_{2}^{\prime }u^{\prime }\left( 1\right) \right) +\left( \beta _{1}u\left(
1\right) -\beta _{2}u^{\prime }\left( 1\right) \right) =0,  \tag{1.3}
\label{equation 3}
\end{equation}%
and the four transmission conditions at the points of discontinuity $x=h_{1}$
and $x=h_{2}$%
\begin{equation}
L_{3}u:=\gamma _{1}u\left( h_{1}-0\right) -\delta _{1}u\left( h_{1}+0\right)
=0,  \tag{1.4}  \label{equation 4}
\end{equation}%
\begin{equation}
L_{4}u:=\gamma _{2}u^{\prime }\left( h_{1}-0\right) -\delta _{2}u^{\prime
}\left( h_{1}+0\right) =0,  \tag{1.5}
\end{equation}%
\begin{equation}
L_{5}u:=\gamma _{3}u\left( h_{2}-0\right) -\delta _{3}u\left( h_{2}+0\right)
=0,  \tag{1.6}
\end{equation}%
\begin{equation}
L_{6}u:=\gamma _{4}u^{\prime }\left( h_{2}-0\right) -\delta _{4}u^{\prime
}\left( h_{2}+0\right) =0,  \tag{1.7}
\end{equation}%
in the Hilbert space $L_{2}\left( -1,h_{1}\right) \oplus L_{2}\left(
h_{1},h_{2}\right) \oplus L_{2}\left( h_{2},1\right) $ where $\lambda \in
\mathbb{C}
$ is a complex spectral parameter; and all coefficients of the boundary and
transmission conditions are real constants. We assume naturally that $%
\left\vert \alpha _{1}\right\vert +\left\vert \alpha _{2}\right\vert \neq 0$%
, $\left\vert \beta _{1}^{\prime }\right\vert +\left\vert \beta _{2}^{\prime
}\right\vert \neq 0$ and $\left\vert \beta _{1}\right\vert +\left\vert \beta
_{2}\right\vert \neq 0$. Moreover, we will assume that $\rho :=\beta
_{1}^{\prime }\beta _{2}-\beta _{1}\beta _{2}^{\prime }>0$. Some special
cases of this problem arises after application of the method of speration of
variables to the diverse assortment of physical problems, heat and mass
transfer problems (for example, see [22]), vibrating string problems when
the string loaded additionally with point masses (for example, see [22]).

\section{\textbf{Operator-Theoretic Formulation of the Problem}}
In the series of O. Sh. Mukhtarov and his students works are
introduced direct sum of Hilbert spaces but with the usual inner
products replaced by appropriate multiplies  (see, for example,
[1-3,5,6,10-12].
By employing the approach used in these words, we introduce a special inner product in the Hilbert space $%
\left( L_{2}\left( -1,h_{1}\right) \oplus L_{2}\left( h_{1},h_{2}\right)
\oplus L_{2}\left( h_{2},1\right) \right) \oplus
\mathbb{C}
$ and define a linear operator $A$ in it so that the problem (1.1)-(1.5) can
be interpreted as the eigenvalue problem for $A$. To this end, we define a
new Hilbert space inner product on $H:=\left( L_{2}\left( -1,h_{1}\right)
\oplus L_{2}\left( h_{1},h_{2}\right) \oplus L_{2}\left( h_{2},1\right)
\right) \oplus
\mathbb{C}
$ by
\begin{eqnarray*}
\left\langle F,G\right\rangle _{H} &=&\omega _{1}^{2}\int_{-1}^{h_{1}}f(x)%
\overline{g(x)}dx+\omega _{2}^{2}\frac{\delta _{1}\delta _{2}}{\gamma
_{1}\gamma _{2}}\int_{h_{1}}^{h_{2}}f(x)\overline{g(x)}dx \\
&&+\omega _{3}^{2}\frac{\delta _{1}\delta _{2}\delta _{3}\delta _{4}}{\gamma
_{1}\gamma _{2}\gamma _{3}\gamma _{4}}\int_{h_{2}}^{1}f(x)\overline{g(x)}dx+%
\frac{\delta _{1}\delta _{2}\delta _{3}\delta _{4}}{\rho \gamma _{1}\gamma
_{2}\gamma _{3}\gamma _{4}}f_{1}\overline{g_{1}}
\end{eqnarray*}%
for $F=\left(
\begin{array}{c}
f(x) \\
f_{1}%
\end{array}%
\right) $ and $G=\left(
\begin{array}{c}
g(x) \\
g_{1}%
\end{array}%
\right) \in H$. For convenience we will use the notations
\begin{equation*}
R_{1}\left( u\right) :=\beta _{1}u(1)-\beta _{2}u^{\prime }(1),\text{ \ }%
R_{1}^{\prime }\left( u\right) :=\beta _{1}^{\prime }u(1)-\beta _{2}^{\prime
}u^{\prime }(1).
\end{equation*}%
In this Hilbert space we construct the operator $A:H\rightarrow H$ with
domain%
\begin{eqnarray}
D(A) &=&\left\{ F=\left(
\begin{array}{c}
f(x) \\
f_{1}%
\end{array}%
\right) \mid f(x),f^{\prime }(x)\text{ are absolutely continuous in }\left[
1,h_{1}\right] \cup \left[ h_{1},h_{2}\right] \right.  \notag \\
&&\cup \left[ h_{2},1\right] \text{; and has finite limits }f(h_{1}\pm 0),%
\text{ }f(h_{2}\pm 0),f^{\prime }(h_{1}\pm 0),f^{\prime }(h_{2}\pm 0);
\notag \\
\tau f &\in &L_{2}\left( -1,h_{1}\right) \oplus L_{2}\left(
h_{1},h_{2}\right) \oplus L_{2}\left( h_{2},1\right) ;\text{ }%
L_{1}f=L_{3}f=L_{4}f=L_{5}f=L_{6}f=0,  \notag \\
&&\left. f_{1}=R_{1}^{\prime }(f)\right\}  \TCItag{2.1}
\end{eqnarray}%
which acts by the rule%
\begin{equation}
AF=\left(
\begin{array}{c}
\frac{1}{\omega \left( x\right) }\left[ -f^{\prime \prime }+q(x)f\right] \\
-R_{1}(f)%
\end{array}%
\right) \text{ \ \ \ with \ }F=\left(
\begin{array}{c}
f(x) \\
R_{1}^{\prime }(f)%
\end{array}%
\right) \in D(A).  \tag{2.2}
\end{equation}%
Thus we can pose the boundary-value-transmission problem (1.1)-(1.7) in $H$
as
\begin{equation}
AU=\lambda U,\text{ \ \ \ }U:=\left(
\begin{array}{c}
u(x) \\
R_{1}^{\prime }(u)%
\end{array}%
\right) \in D(A).  \tag{2.3}
\end{equation}%
It is readily verified that the eigenvalues of $A$ coincide with those of
the problem (1.1)-(1.7).

\begin{theorem}
The operator $A$ is symmetric.
\end{theorem}

\begin{proof}
Let $F=\left(
\begin{array}{c}
f(x) \\
R_{1}^{\prime }(f)%
\end{array}%
\right) $ and $G=\left(
\begin{array}{c}
g(x) \\
R_{1}^{\prime }(g)%
\end{array}%
\right) $ be arbitrary elements of $D(A)$. Twice integrating by parts we find%
\begin{equation*}
\left\langle AF,G\right\rangle _{H}-\left\langle F,AG\right\rangle
_{H}=W\left( f,\overline{g};h_{1}-0\right) -W\left( f,\overline{g};-1\right)
\end{equation*}%
\begin{eqnarray}
&&+\frac{\delta _{1}\delta _{2}}{\gamma _{1}\gamma _{2}}\left( W\left( f,%
\overline{g};h_{2}-0\right) -W\left( f,\overline{g};h_{1}+0\right) \right)
\notag \\
&&+\frac{\delta _{1}\delta _{2}\delta _{3}\delta _{4}}{\gamma _{1}\gamma
_{2}\gamma _{3}\gamma _{4}}\left( W\left( f,\overline{g};1\right) -W\left( f,%
\overline{g};h_{2}+0\right) \right)  \notag \\
&&+\frac{\delta _{1}\delta _{2}\delta _{3}\delta _{4}}{\rho \gamma
_{1}\gamma _{2}\gamma _{3}\gamma _{4}}\left( R_{1}^{\prime }(f)R_{1}(%
\overline{g})-R_{1}(f)R_{1}^{\prime }(\overline{g})\right)  \TCItag{2.4}
\end{eqnarray}%
where, as usual, $W\left( f,g;x\right) $ denotes the Wronskian of $f$ and $g$%
; i.e.,%
\begin{equation*}
W\left( f,g;x\right) :=f(x)g^{\prime }(x)-f^{\prime }(x)g(x).
\end{equation*}%
Since $F,G\in D(A),$ the first components of these elements, i.e. $f$ and $g$
satisfy the boundary condition (1.2). From this fact we easily see that%
\begin{equation}
W\left( f,\overline{g};-1\right) =0,  \tag{2.5}
\end{equation}%
since $\cos \alpha $ and $\sin \alpha $ are real. Further, as $f$ and $g$
also satisfy both transmission conditions, we obtain%
\begin{equation}
W\left( f,\overline{g};h_{1}-0\right) =\frac{\delta _{1}\delta _{2}}{\gamma
_{1}\gamma _{2}}W\left( f,\overline{g};h_{1}+0\right)  \tag{2.6}
\end{equation}%
\begin{equation}
W\left( f,\overline{g};h_{2}-0\right) =\frac{\delta _{1}\delta _{2}\delta
_{3}\delta _{4}}{\gamma _{1}\gamma _{2}\gamma _{3}\gamma _{4}}W\left( f,%
\overline{g};h_{2}+0\right)  \tag{2.7}
\end{equation}%
Moreover, the direct calculations give%
\begin{equation}
R_{1}^{\prime }(f)R_{1}(\overline{g})-R_{1}(f)R_{1}^{\prime }(\overline{g}%
)=-\rho W\left( f,\overline{g};1\right)  \tag{2.8}
\end{equation}%
Now, inserting (2.5)-(2.8) in (2.4), we have%
\begin{equation*}
\left\langle AF,G\right\rangle _{H}=\left\langle F,AG\right\rangle _{H}\text{
\ \ \ \ }\left( F,G\in D(A\right)
\end{equation*}%
and so $A$ is symmetric.
\end{proof}

Recalling that the eigenvalues of (1.1)-(1.7) coincide with the eigenvalues
of $A$, we have the next corollary:

\begin{corollary}
All eigenvalues of (1.1)-(1.7) are real.
\end{corollary}

Since all eigenvalues are real it is enough to study only the real-valued
eigenfunctions. Therefore we can now assume that all eigenfunctions of
(1.1)-(1.7) are real-valued.

\section{\textbf{Asymptotic Formulas for Eigenvalues and Fundamental
Solutions}}

Let us define fundamental solutions%
\begin{equation*}
\phi \left( x,\lambda \right) =\left\{
\begin{array}{cc}
\phi _{1}\left( x,\lambda \right) , & x\in \left[ -1,h_{1}\right) , \\
\phi _{2}\left( x,\lambda \right) , & x\in \left( h_{1},h_{2}\right) , \\
\phi _{3}\left( x,\lambda \right) , & x\in \left( h_{2},1\right]%
\end{array}%
\right. \text{ and\ }\chi \left( x,\lambda \right) =\left\{
\begin{array}{cc}
\chi _{1}\left( x,\lambda \right) , & x\in \left[ -1,h_{1}\right) , \\
\chi _{2}\left( x,\lambda \right) , & x\in \left( h_{1},h_{2}\right) , \\
\chi _{3}\left( x,\lambda \right) , & x\in \left( h_{2},1\right]%
\end{array}%
\right.
\end{equation*}%
of (1.1) by the following procedure. We first consider the next
initial-value problem:%
\begin{equation}
-u^{\prime \prime }+q\left( x\right) u=\lambda \omega _{1}^{2}u,\ x\in \left[
-1,h_{1}\right]  \tag{3.1}
\end{equation}%
\begin{eqnarray}
u(-1) &=&\sin \alpha ,  \TCItag{3.2} \\
u^{\prime }(-1) &=&-\cos \alpha  \TCItag{3.3}
\end{eqnarray}%
By virtue of [14, Theorem 1.5] the problem (3.1)-(3.3) has a unique solution
$u=\phi _{1}\left( x,\lambda \right) $ which is an entire function of $%
\lambda \in
\mathbb{C}
$ for each fixed $x\in \left[ -1,h_{1}\right] $. Similarly,

\begin{equation}
-u^{\prime \prime }+q\left( x\right) u=\lambda \omega _{2}^{2}u,\text{ \ }%
x\in \left[ h_{1},h_{2}\right]  \tag{3.4}
\end{equation}%
\begin{eqnarray}
u(h_{1}) &=&\frac{\gamma _{1}}{\delta _{1}}\phi _{1}\left( h_{1},\lambda
\right) ,  \TCItag{3.5} \\
u^{\prime }(h_{1}) &=&\frac{\gamma _{2}}{\delta _{2}}\phi _{1}^{\prime
}\left( h_{1},\lambda \right) ,  \TCItag{3.6}
\end{eqnarray}%
has a unique solution $u=\phi _{2}\left( x,\lambda \right) $ which is an
entire function of $\lambda \in
\mathbb{C}
$ for each fixed $x\in \left[ h_{1},h_{2}\right] $. Continuing in this manner

\begin{equation}
-u^{\prime \prime }+q\left( x\right) u=\lambda \omega _{3}^{2}u,\text{ \ }%
x\in \left[ h_{2},1\right]  \tag{3.7}
\end{equation}%
\begin{eqnarray}
u(h_{2}) &=&\frac{\gamma _{3}}{\delta _{3}}\phi _{2}\left( h_{2},\lambda
\right) ,  \TCItag{3.8} \\
u^{\prime }(h_{2}) &=&\frac{\gamma _{4}}{\delta _{4}}\phi _{2}^{\prime
}\left( h_{2},\lambda \right) ,  \TCItag{3.9}
\end{eqnarray}%
has a unique solution $u=\phi _{3}\left( x,\lambda \right) $ which is an
entire function of $\lambda \in
\mathbb{C}
$ for each fixed $x\in \left[ h_{2},1\right] $. Slightly \ modifying the
method of \ [2, Theorem 1.5] we can prove that the initial-value problem

\begin{equation}
-u^{\prime \prime }+q\left( x\right) u=\lambda \omega _{3}^{2}u,\ x\in \left[
h_{2},1\right]  \tag{3.10}
\end{equation}%
\begin{eqnarray}
u(1) &=&\beta _{2}^{\prime }\lambda +\beta _{2},  \TCItag{3.11} \\
u^{\prime }(1) &=&\beta _{1}^{\prime }\lambda +\beta _{1}  \TCItag{3.12}
\end{eqnarray}%
(3.10)-(3.13) has a unique solution $u=\chi _{3}\left( x,\lambda \right) $
which is an entire function of spectral parameter $\lambda \in
\mathbb{C}
$ for each fixed $x\in \left[ h_{2},1\right] $. Similarly,

\begin{equation}
-u^{\prime \prime }+q\left( x\right) u=\lambda \omega _{2}^{2}u,\text{ \ }%
x\in \left[ h_{1},h_{2}\right]  \tag{3.13}
\end{equation}%
\begin{eqnarray}
u(h_{2}) &=&\frac{\delta _{3}}{\gamma _{3}}\chi _{3}\left( h_{2},\lambda
\right) ,  \TCItag{3.14} \\
u^{\prime }(h_{2}) &=&\frac{\delta _{4}}{\gamma _{4}}\chi _{3}^{\prime
}\left( h_{2},\lambda \right) ,  \TCItag{3.15}
\end{eqnarray}%
has a unique solution $u=\chi _{2}\left( x,\lambda \right) $ which is an
entire function of $\lambda \in
\mathbb{C}
$ for each fixed $x\in \left[ h_{1},h_{2}\right] $. Continuing in this manner

\begin{equation}
-u^{\prime \prime }+q\left( x\right) u=\lambda \omega _{3}^{2}u,\text{ \ }%
x\in \left[ -1,h_{1}\right]  \tag{3.16}
\end{equation}%
\begin{eqnarray}
u(h_{1}) &=&\frac{\delta _{1}}{\gamma _{1}}\chi _{2}\left( h_{1},\lambda
\right) ,  \TCItag{3.17} \\
u^{\prime }(h_{1}) &=&\frac{\delta _{2}}{\gamma _{2}}\chi _{2}^{\prime
}\left( h_{1},\lambda \right) ,  \TCItag{3.18}
\end{eqnarray}%
has a unique solution $u=\chi _{1}\left( x,\lambda \right) $ which is an
entire function of $\lambda \in
\mathbb{C}
$ for each fixed $x\in \left[ -1,h_{1}\right] $.

By virtue of (3.2) and (3.3) the solution $\phi \left( x,\lambda \right) $
satisfies the first boundary condition (1.2). Moreover, by (3.5), (3.6),
(3.8) and (3.9), $\phi \left( x,\lambda \right) $ satisfies also
transmission conditions (1.4)-(1.7). Similarly, by (3.11), (3.12), (3.14),
(3.15), (3.17) and (3.18) the other solution $\chi \left( x,\lambda \right) $
satisfies the second boundary condition (1.3) and transmission conditions
(1.4)-(1.7). It is well-known from the theory of ordinary differential
equations that each of the Wronskians $\Delta _{1}\left( \lambda \right)
=W\left( \phi _{1}\left( x,\lambda \right) ,\chi _{1}\left( x,\lambda
\right) \right) ,$ $\Delta _{2}\left( \lambda \right) =W\left( \phi
_{2}\left( x,\lambda \right) ,\chi _{2}\left( x,\lambda \right) \right) $
and $\Delta _{3}\left( \lambda \right) =W\left( \phi _{3}\left( x,\lambda
\right) ,\chi _{3}\left( x,\lambda \right) \right) $ are independent of $x$
in $\left[ -1,h_{1}\right] ,$ $\left[ h_{1},h_{2}\right] $ and $\left[
h_{2},1\right] $ respectively.

\begin{lemma}
The equality $\Delta _{1}\left( \lambda \right) =\frac{\delta _{1}\delta _{2}%
}{\gamma _{1}\gamma _{2}}\Delta _{2}\left( \lambda \right) =\frac{\delta
_{1}\delta _{2}\delta _{3}\delta _{4}}{\gamma _{1}\gamma _{2}\gamma
_{3}\gamma _{4}}\Delta _{3}\left( \lambda \right) $ holds for each $\lambda
\in
\mathbb{C}
$.
\end{lemma}

\begin{proof}
Since the above Wronskians are independent of $x$, using (3.8), (3.9),
(3.11), (3.12), (3.14), (3.15), (3.17) and (3.18) we find%
\begin{eqnarray*}
\Delta _{1}\left( \lambda \right) &=&\phi _{1}\left( h_{1},\lambda \right)
\chi _{1}^{\prime }\left( h_{1},\lambda \right) -\phi _{1}^{\prime }\left(
h_{1},\lambda \right) \chi _{1}\left( h_{1},\lambda \right) \\
&=&\left( \frac{\delta _{1}}{\gamma _{1}}\phi _{2}\left( h_{1},\lambda
\right) \right) \left( \frac{\delta _{2}}{\gamma _{2}}\chi _{2}^{\prime
}\left( h_{1},\lambda \right) \right) -\left( \frac{\delta _{2}}{\gamma _{2}}%
\phi _{2}^{\prime }\left( h_{1},\lambda \right) \right) \left( \frac{\delta
_{1}}{\gamma _{1}}\chi _{2}\left( h_{1},\lambda \right) \right) \\
&=&\frac{\delta _{1}\delta _{2}}{\gamma _{1}\gamma _{2}}\Delta _{2}\left(
\lambda \right) =\left( \frac{\delta _{1}\delta _{3}}{\gamma _{1}\gamma _{3}}%
\phi _{3}\left( h_{2},\lambda \right) \right) \left( \frac{\delta _{2}\delta
_{4}}{\gamma _{2}\gamma _{4}}\chi _{3}^{\prime }\left( h_{2},\lambda \right)
\right)
\end{eqnarray*}%
\begin{equation*}
-\left( \frac{\delta _{2}\delta _{4}}{\gamma _{2}\gamma _{4}}\phi
_{3}^{\prime }\left( h_{2},\lambda \right) \right) \left( \frac{\delta
_{1}\delta _{3}}{\gamma _{1}\gamma _{3}}\chi _{3}\left( h_{2},\lambda
\right) \right) =\frac{\delta _{1}\delta _{2}\delta _{3}\delta _{4}}{\gamma
_{1}\gamma _{2}\gamma _{3}\gamma _{4}}\Delta _{3}\left( \lambda \right) .
\end{equation*}
\end{proof}

\begin{corollary}
The zeros of $\Delta _{1}\left( \lambda \right) ,$ $\Delta _{2}\left(
\lambda \right) $ and $\Delta _{3}\left( \lambda \right) $ coincide.
\end{corollary}

In view of Lemma 3.1 we denote $\Delta _{1}\left( \lambda \right) ,$ $\frac{%
\delta _{1}\delta _{2}}{\gamma _{1}\gamma _{2}}\Delta _{2}\left( \lambda
\right) $ and $\frac{\delta _{1}\delta _{2}\delta _{3}\delta _{4}}{\gamma
_{1}\gamma _{2}\gamma _{3}\gamma _{4}}\Delta _{3}\left( \lambda \right) $\
by $\Delta \left( \lambda \right) $. Recalling the definitions of $\phi
_{i}\left( x,\lambda \right) $ and $\chi _{i}\left( x,\lambda \right) $, we
infer the next corollary.

\begin{corollary}
The function $\Delta \left( \lambda \right) $ is an entire function.
\end{corollary}

\begin{theorem}
The eigenvalues of (1.1)-(1.7) coincide with the zeros of $\Delta \left(
\lambda \right) $.
\end{theorem}

\begin{proof}
Let $\Delta \left( \lambda _{0}\right) =0.$ Then $W\left( \phi _{1}\left(
x,\lambda _{0}\right) ,\chi _{1}\left( x,\lambda _{0}\right) \right) =0$ for
all $x\in \left[ -1,h_{1}\right] .$ Consequently, the functions $\phi
_{1}\left( x,\lambda _{0}\right) $ and $\chi _{1}\left( x,\lambda
_{0}\right) $ are linearly dependent, i.e.,$\chi _{1}\left( x,\lambda
_{0}\right) =k\phi _{1}\left( x,\lambda _{0}\right) $, $x\in \left[ -1,h_{1}%
\right] $, for some $k\neq 0.$ By (3.2) and (3.3), from this equality, we
have%
\begin{equation*}
\cos \alpha \chi \left( -1,\lambda _{0}\right) +\sin \alpha \chi ^{\prime
}\left( -1,\lambda _{0}\right) =\cos \alpha \chi _{1}\left( -1,\lambda
_{0}\right) +\sin \alpha \chi _{1}^{\prime }\left( -1,\lambda _{0}\right)
\end{equation*}%
\begin{equation*}
=k\left( \cos \alpha \phi _{1}\left( -1,\lambda _{0}\right) +\sin \alpha
\phi _{1}^{\prime }\left( -1,\lambda _{0}\right) \right) =k\left( \cos
\alpha \sin \alpha +\sin \alpha \left( -\cos \alpha \right) \right) =0,
\end{equation*}%
and so $\chi \left( x,\lambda _{0}\right) $ satisfies the first boundary
condition (1.2). Recalling that the solution $\chi \left( x,\lambda
_{0}\right) $ also satisfies the other boundary condition (1.3) and
transmission conditions (1.4)-(1.7). We conclude that $\chi \left( x,\lambda
_{0}\right) $ is an eigenfunction of (1.1)-(1.7); i.e., $\lambda _{0}$ is an
eigenvalue. Thus, each zero of $\Delta \left( \lambda \right) $ is an
eigenvalue. Now let $\lambda _{0}$ be an eigenvalue and let $u_{0}\left(
x\right) $ be an eigenfunction with this eigenvalue. Suppose that $\Delta
\left( \lambda _{0}\right) \neq 0$. Whence $W\left( \phi _{1}\left(
x,\lambda _{0}\right) ,\chi _{1}\left( x,\lambda _{0}\right) \right) \neq 0$%
, $W\left( \phi _{2}\left( x,\lambda _{0}\right) ,\chi _{2}\left( x,\lambda
_{0}\right) \right) \neq 0$ and $W\left( \phi _{3}\left( x,\lambda
_{0}\right) ,\chi _{3}\left( x,\lambda _{0}\right) \right) \neq 0$. From
this, by virtue of the well-known properties of Wronskians, it follows that
each of the pairs $\phi _{1}\left( x,\lambda _{0}\right) ,$ $\chi _{1}\left(
x,\lambda _{0}\right) $; $\phi _{2}\left( x,\lambda _{0}\right) ,$ $\chi
_{2}\left( x,\lambda _{0}\right) $ and $\phi _{3}\left( x,\lambda
_{0}\right) ,$ $\chi _{3}\left( x,\lambda _{0}\right) $ is linearly
independent. Therefore, the solution $u_{0}(x)$ of (1.1) may be represented
as
\begin{equation*}
u_{0}\left( x\right) =\left\{
\begin{array}{c}
c_{1}\phi _{1}\left( x,\lambda _{0}\right) +c_{2}\chi _{1}\left( x,\lambda
_{0}\right) ,\text{ \ }x\in \left[ -1,h_{1}\right) , \\
c_{3}\phi _{2}\left( x,\lambda _{0}\right) +c_{4}\chi _{2}\left( x,\lambda
_{0}\right) ,\text{ \ }x\in \left( h_{1},h_{2}\right) , \\
c_{5}\phi _{3}\left( x,\lambda _{0}\right) +c_{6}\chi _{3}\left( x,\lambda
_{0}\right) ,\text{ \ }x\in \left( h_{2},1\right] ,%
\end{array}%
\right.
\end{equation*}%
where at least one of the coefficients $c_{i}$ $\left( i=\overline{1,6}%
\right) $ is not zero. Considering the true equalities%
\begin{equation}
L_{\upsilon }\left( u_{0}\left( x\right) \right) =0,\text{ \ \ }\upsilon =%
\overline{1,6},  \tag{3.19}
\end{equation}%
as the homogenous system of linear equations in the variables $c_{i}$ $%
\left( i=\overline{1,6}\right) $ and taking (3.5), (3.6), (3.8), (3.9),
(3.14), (3.15), (3.17) and (3.18) into account, we see that the determinant
of this system is equal to $-\frac{\left( \delta _{1}\delta _{2}\delta
_{3}\delta _{4}\right) ^{2}}{\gamma _{1}\gamma _{2}\gamma _{3}\gamma _{4}}%
\Delta ^{4}\left( \lambda _{0}\right) $ and so it does not vanish by
assumption. Consequently the system (3.19) has the only trivial solution $%
c_{i}=0$ $\left( i=\overline{1,6}\right) $. We thus get at a contradiction,
which completes the proof.
\end{proof}

\begin{theorem}
Let $\lambda =\mu ^{2}$ and $\func{Im}\mu =t$. Then the following asymptotic
equalities hold as $\left\vert \lambda \right\vert \rightarrow \infty :$

(1) In case $\sin \alpha \neq 0$%
\begin{equation}
\phi _{1}^{\left( k\right) }\left( x,\lambda \right) =\sin \alpha \frac{d^{k}%
}{dx^{k}}\cos \left[ \mu \omega _{1}\left( x+1\right) \right] +O\left( \frac{%
1}{\left\vert \mu \right\vert ^{1-k}}\exp \left( \left\vert t\right\vert
\omega _{1}\left( x+1\right) \right) \right) ,  \tag{3.20}
\end{equation}%
\begin{eqnarray}
\phi _{2}^{\left( k\right) }\left( x,\lambda \right) &=&\frac{\gamma _{1}}{%
\delta _{1}}\sin \alpha \frac{d^{k}}{dx^{k}}\cos \left[ \mu \left( \omega
_{2}x+\omega _{1}h_{1}+\omega _{1}\right) \right]  \notag \\
&&+O\left( \frac{1}{\left\vert \mu \right\vert ^{1-k}}\exp \left( \left\vert
t\right\vert \left( \omega _{2}x+\omega _{1}h_{1}+\omega _{1}\right) \right)
\right) ,  \TCItag{3.21}
\end{eqnarray}%
\begin{eqnarray}
\phi _{3}^{\left( k\right) }\left( x,\lambda \right) &=&\frac{\gamma
_{1}\gamma _{3}}{\delta _{1}\delta _{3}}\sin \alpha \frac{d^{k}}{dx^{k}}\cos %
\left[ \mu \left( \omega _{3}x+\omega _{2}h_{2}+\omega _{1}\right) \right]
\notag \\
&&+O\left( \frac{1}{\left\vert \mu \right\vert ^{1-k}}\exp \left( \left\vert
t\right\vert \left( \omega _{3}x+\omega _{2}h_{2}+\omega _{1}\right) \right)
\right) .  \TCItag{3.22}
\end{eqnarray}%
(1) In case $\sin \alpha =0$%
\begin{equation}
\phi _{1}^{\left( k\right) }\left( x,\lambda \right) =\frac{-1}{\mu \omega
_{1}}\cos \alpha \frac{d^{k}}{dx^{k}}\sin \left[ \mu \omega _{1}\left(
x+1\right) \right] +O\left( \frac{1}{\left\vert \mu \right\vert ^{2-k}}\exp
\left( \left\vert t\right\vert \omega _{1}\left( x+1\right) \right) \right) ,
\tag{3.23}
\end{equation}%
\begin{eqnarray}
\phi _{2}^{\left( k\right) }\left( x,\lambda \right) &=&-\frac{\gamma _{1}}{%
\mu \delta _{1}}\cos \alpha \frac{d^{k}}{dx^{k}}\sin \left[ \mu \left(
\omega _{2}x+\omega _{1}h_{1}+\omega _{1}\right) \right]  \notag \\
&&+O\left( \frac{1}{\left\vert \mu \right\vert ^{2-k}}\exp \left( \left\vert
t\right\vert \left( \omega _{2}x+\omega _{1}h_{1}+\omega _{1}\right) \right)
\right) ,  \TCItag{3.24}
\end{eqnarray}%
\begin{eqnarray}
\phi _{3}^{\left( k\right) }\left( x,\lambda \right) &=&-\frac{\gamma
_{1}\gamma _{3}}{\mu \delta _{1}\delta _{3}}\cos \alpha \frac{d^{k}}{dx^{k}}%
\sin \left[ \mu \left( \omega _{3}x+\omega _{2}h_{2}+\omega _{1}\right) %
\right]  \notag \\
&&+O\left( \frac{1}{\left\vert \mu \right\vert ^{2-k}}\exp \left( \left\vert
t\right\vert \left( \omega _{3}x+\omega _{2}h_{2}+\omega _{1}\right) \right)
\right) .  \TCItag{3.25}
\end{eqnarray}%
for $k=0$ and $k=1$. Moreover, each of these asymptotic equalities holds
uniformly for $x$.
\end{theorem}

\begin{proof}
Asymptotic formulas for $\phi _{1}\left( x,\lambda \right) $ and $\phi
_{2}\left( x,\lambda \right) $ are found in [18, Lemma 1.7] and [12, Theorem
3.2] respectively. But the formulas for $\phi _{3}\left( x,\lambda \right) $
need individual considerations, since this solution is defined by the
initial condition with some special nonstandart form. The initial-value
problem (3.7)-(3.9) can be transformed into the equivalent integral equation%
\begin{eqnarray}
u(x) &=&\frac{\gamma _{3}}{\delta _{3}}\phi _{2}\left( h_{2},\lambda \right)
\cos \mu \omega _{3}x+\frac{\gamma _{4}}{\mu \omega _{3}\delta _{4}}\phi
_{2}^{\prime }\left( h_{2},\lambda \right) \sin \mu \omega _{3}x  \notag \\
&&+\frac{\omega _{3}}{\mu }\int_{h_{2}}^{x}\sin \left[ \mu \omega _{3}\left(
x-y\right) \right] q\left( y\right) u\left( y\right) dy  \TCItag{3.26}
\end{eqnarray}%
Let $\sin \alpha \neq 0$. Inserting (3.21) in (3.26) we have%
\begin{eqnarray}
\phi _{3}\left( x,\lambda \right) &=&\frac{\gamma _{1}\gamma _{3}}{\delta
_{1}\delta _{3}}\sin \alpha \cos \left[ \mu \left( \omega _{3}x+\omega
_{2}h_{2}+\omega _{1}\right) \right]  \notag \\
&&+\frac{\omega _{3}}{\mu }\int_{h_{2}}^{x}\sin \left[ \mu \omega _{3}\left(
x-y\right) \right] q\left( y\right) \phi _{3}\left( y,\lambda \right) dy
\notag \\
&&+O\left( \frac{1}{\left\vert \mu \right\vert }\exp \left( \left\vert
t\right\vert \left( \omega _{3}x+\omega _{2}h_{2}+\omega _{1}\right) \right)
\right) .  \TCItag{3.27}
\end{eqnarray}%
Multiplying this by $\exp \left( -\left\vert t\right\vert \left( \omega
_{3}x+\omega _{2}h_{2}+\omega _{1}\right) \right) $ and denoting $%
F(x,\lambda )=\exp \left( -\left\vert t\right\vert \left( \omega
_{3}x+\omega _{2}h_{2}+\omega _{1}\right) \right) \phi _{3}\left( x,\lambda
\right) $, we have the next "asymptotic integral equation"%
\begin{eqnarray*}
F(x,\lambda ) &=&\frac{\gamma _{1}\gamma _{3}}{\delta _{1}\delta _{3}}\sin
\alpha \exp \left( -\left\vert t\right\vert \left( \omega _{3}x+\omega
_{2}h_{2}+\omega _{1}\right) \right) \cos \left[ \mu \left( \omega
_{3}x+\omega _{2}h_{2}+\omega _{1}\right) \right] \\
&&+\frac{\omega _{3}}{\mu }\int_{h_{2}}^{x}\sin \left[ \mu \omega _{3}\left(
x-y\right) \right] \exp \left( -\left\vert t\right\vert \omega _{3}\left(
x-y\right) \right) q\left( y\right) F(y,\lambda )dy+O\left( \frac{1}{\mu }%
\right) .
\end{eqnarray*}%
Putting $M(\lambda )=\max_{x\in \left[ h_{2},1\right] }\left\vert
F(x,\lambda )\right\vert $, from the last equation we derive that

\begin{equation*}
M(\lambda )\leq M_{0}\left( \left\vert \frac{\gamma _{1}\gamma _{3}}{\delta
_{1}\delta _{3}}\right\vert +\frac{1}{\mu }\right)
\end{equation*}%
for some $M_{0}>0$. Consequently, $M(\lambda )=O\left( 1\right) $ as $%
\left\vert \lambda \right\vert \rightarrow \infty $, and so $\phi _{3}\left(
x,\lambda \right) =O\left( \exp \left( \left\vert t\right\vert \left( \omega
_{3}x+\omega _{2}h_{2}+\omega _{1}\right) \right) \right) $ as $\left\vert
\lambda \right\vert \rightarrow \infty $. Inserting the integral term of
(3.27) yields (3.22) for $k=0$. The case $k=1$ of (3.22) follows at once on
differentiating (3.21) and making the same procedure as in the case $k=0$.
The proof of (3.25) is similar to that of (3.22).
\end{proof}

\begin{theorem}
Let $\lambda =\mu ^{2}$, $\mu =\sigma +it$. Then the following asymptotic
formulas hold for the eigenvalues of the boundary-value-transmission problem
(1.1)-(1.7):

Case 1: $\beta _{2}^{\prime }\neq 0$, $\sin \alpha \neq 0$%
\begin{equation}
\mu _{n}=\frac{\pi \left( n-1\right) }{\omega _{3}+\omega _{2}h_{2}+\omega
_{1}}+O\left( \frac{1}{n}\right) ,  \tag{3.28}
\end{equation}%
Case 2: $\beta _{2}^{\prime }\neq 0$, $\sin \alpha =0$%
\begin{equation}
\mu _{n}=\frac{\pi \left( n-\frac{1}{2}\right) }{\omega _{3}+\omega
_{2}h_{2}+\omega _{1}}+O\left( \frac{1}{n}\right) ,  \tag{3.29}
\end{equation}%
Case 3: $\beta _{2}^{\prime }=0$, $\sin \alpha \neq 0$%
\begin{equation}
\mu _{n}=\frac{\pi \left( n-\frac{1}{2}\right) }{\omega _{3}+\omega
_{2}h_{2}+\omega _{1}}+O\left( \frac{1}{n}\right) ,  \tag{3.30}
\end{equation}%
Case 4: $\beta _{2}^{\prime }=0$, $\sin \alpha =0$%
\begin{equation}
\mu _{n}=\frac{\pi n}{\omega _{3}+\omega _{2}h_{2}+\omega _{1}}+O\left(
\frac{1}{n}\right) ,  \tag{3.31}
\end{equation}
\end{theorem}

\begin{proof}
Let us consider only the case 1. Putting $x=1$ in $\Delta _{3}\left( \lambda
\right) =\phi _{3}\left( x,\lambda \right) \chi _{3}^{\prime }\left(
x,\lambda \right) -\phi _{3}^{\prime }\left( x,\lambda \right) \chi
_{3}\left( x,\lambda \right) $ and inserting $\chi _{3}\left( 1,\lambda
\right) =\beta _{2}^{\prime }\lambda +\beta _{2},$ $\chi _{3}^{\prime
}\left( 1,\lambda \right) =\beta _{1}^{\prime }\lambda +\beta _{1}$ we have
the following representation for $\Delta _{3}\left( \lambda \right) $:%
\begin{equation}
\Delta _{3}\left( \lambda \right) =\left( \beta _{1}^{\prime }\lambda +\beta
_{1}\right) \phi _{3}\left( 1,\lambda \right) -\left( \beta _{2}^{\prime
}\lambda +\beta _{2}\right) \phi _{3}^{\prime }\left( 1,\lambda \right) .
\tag{3.32}
\end{equation}%
Putting $x=1$ in (3.22) and inserting the result in (3.32), we derive now
that%
\begin{eqnarray}
\Delta _{3}\left( \lambda \right) &=&\frac{\delta _{2}\delta _{4}}{\gamma
_{2}\gamma _{4}}\omega _{3}\beta _{2}^{\prime }\left( \sin \alpha \right)
\mu ^{3}\sin \left[ \mu \left( \omega _{3}+\omega _{2}h_{2}+\omega
_{1}\right) \right]  \notag \\
&&+O\left( \left\vert \mu \right\vert ^{2}\exp \left( 2\left\vert
t\right\vert \left( \omega +\omega _{2}h_{2}+\omega _{1}\right) \right)
\right) .  \TCItag{3.33}
\end{eqnarray}%
By applying the well-known Rouch\'{e} Theorem which asserts that if $f\left(
z\right) $ and $g(z)$ are analytic inside and on a closed contour $\Gamma $,
and $\left\vert g(z)\right\vert <\left\vert f(z)\right\vert $ on $\Gamma $
then $f(z)$ and $f(z)+g(z)$ have the same number of zeros inside $\Gamma $
provided that the zeros are counted with multiplicity on a sufficiently
large contour, it follows that $\Delta _{3}\left( \lambda \right) $ has the
same number of zeros inside the contour as the leading term in (3.33).
Hence, if $\lambda _{0}<\lambda _{1}<\lambda _{2}...$ are the zeros of $%
\Delta _{3}\left( \lambda \right) $ and $\mu _{n}^{2}=\lambda _{n},$ we have
\begin{equation}
\frac{\pi \left( n-1\right) }{\omega _{3}+\omega _{2}h_{2}+\omega _{1}}%
+\delta _{n}  \tag{3.34}
\end{equation}%
for sufficiently large $n$, where $\left\vert \delta _{n}\right\vert <\frac{%
\pi }{4\left( \omega _{3}+\omega _{2}h_{2}+\omega _{1}\right) }$ for
sufficiently large $n$. By putting in (3.33) we have $\delta _{n}=O\left(
\frac{1}{n}\right) $, and the proof is completed in Case 1. The proofs for
the other cases are similar.
\end{proof}

\begin{theorem}
The following asymptotic formulas hold for the eigenfunctions
\begin{equation*}
\phi _{\lambda _{n}}\left( x\right) =\left\{
\begin{array}{cc}
\phi _{1}\left( x,\lambda _{n}\right) , & x\in \left[ -1,h_{1}\right) , \\
\phi _{2}\left( x,\lambda _{n}\right) , & x\in \left( h_{1},h_{2}\right) ,
\\
\phi _{3}\left( x,\lambda _{n}\right) , & x\in \left( h_{2},1\right]%
\end{array}%
\right.
\end{equation*}%
of (1.1)-(1.7):

Case 1: $\beta _{2}^{\prime }\neq 0$, $\sin \alpha \neq 0$%
\begin{equation*}
\phi _{\lambda _{n}}\left( x\right) =\left\{
\begin{array}{c}
\sin \alpha \cos \left[ \frac{\omega _{1}\pi \left( n-1\right) \left(
x+1\right) }{\omega _{2}+\omega _{1}}\right] +O\left( \frac{1}{n}\right) ,%
\text{\ }x\in \left[ -1,h_{1}\right) , \\
\frac{\gamma _{1}}{\delta _{1}}\sin \alpha \cos \left[ \frac{\left( \omega
_{2}x+\omega _{1}h_{1}+\omega _{1}\right) \pi \left( n-1\right) }{\omega
_{2}+\omega _{1}h_{1}+\omega _{1}}\right] +O\left( \frac{1}{n}\right) ,\text{%
\ }x\in \left( h_{1},h_{2}\right) , \\
\frac{\gamma _{1}\gamma _{3}}{\delta _{1}\delta _{3}}\sin \alpha \cos \left[
\frac{\left( \omega _{3}x+\omega _{2}h_{2}+\omega _{1}\right) \pi \left(
n-1\right) }{\omega _{3}+\omega _{2}h_{2}+\omega _{1}}\right] +O\left( \frac{%
1}{n}\right) ,\text{\ }x\in \left( h_{2},1\right] .%
\end{array}%
\right.
\end{equation*}%
Case 2: $\beta _{2}^{\prime }\neq 0$, $\sin \alpha =0$%
\begin{equation*}
\phi _{\lambda _{n}}\left( x\right) =\left\{
\begin{array}{c}
-\frac{\omega _{1}+\omega _{2}}{\omega _{1}}\frac{\cos \alpha }{\pi \left( n-%
\frac{1}{2}\right) }\sin \left[ \frac{\omega _{1}\pi \left( n-\frac{1}{2}%
\right) \left( x+1\right) }{\omega _{2}+\omega _{1}}\right] +O\left( \frac{1%
}{n^{2}}\right) ,\text{\ }x\in \left[ -1,h_{1}\right) , \\
\frac{-\gamma _{1}}{\delta _{1}}\frac{\omega _{1}+\omega _{2}}{\omega _{1}}%
\frac{\cos \alpha }{\pi \left( n-\frac{1}{2}\right) }\sin \left[ \frac{%
\left( \omega _{2}x+\omega _{1}h_{1}+\omega _{1}\right) \pi \left( n-\frac{1%
}{2}\right) }{\omega _{2}+\omega _{1}h_{1}+\omega _{1}}\right] +O\left(
\frac{1}{n^{2}}\right) ,\text{ }x\in \left( h_{1},h_{2}\right) , \\
\frac{-\gamma _{1}\gamma _{3}}{\delta _{1}\delta _{3}}\frac{\omega
_{1}+\omega _{2}}{\omega _{1}}\frac{\cos \alpha }{\pi \left( n-\frac{1}{2}%
\right) }\sin \left[ \frac{\left( \omega _{3}x+\omega _{2}h_{2}+\omega
_{1}\right) \pi \left( n-\frac{1}{2}\right) }{\omega _{3}+\omega
_{2}h_{2}+\omega _{1}}\right] +O\left( \frac{1}{n^{2}}\right) ,\text{ }x\in
\left( h_{2},1\right] .%
\end{array}%
\right.
\end{equation*}%
Case 3: $\beta _{2}^{\prime }=0$, $\sin \alpha \neq 0$%
\begin{equation*}
\phi _{\lambda _{n}}\left( x\right) =\left\{
\begin{array}{c}
\sin \alpha \cos \left[ \frac{\omega _{1}\pi \left( n-\frac{1}{2}\right)
\left( x+1\right) }{\omega _{2}+\omega _{1}}\right] +O\left( \frac{1}{n}%
\right) ,\text{\ }x\in \left[ -1,h_{1}\right) , \\
\frac{\gamma _{1}}{\delta _{1}}\sin \alpha \cos \left[ \frac{\left( \omega
_{2}x+\omega _{1}h_{1}+\omega _{1}\right) \pi \left( n-\frac{1}{2}\right) }{%
\omega _{2}+\omega _{1}h_{1}+\omega _{1}}\right] +O\left( \frac{1}{n}\right)
,\text{\ }x\in \left( h_{1},h_{2}\right) , \\
\frac{\gamma _{1}\gamma _{3}}{\delta _{1}\delta _{3}}\sin \alpha \cos \left[
\frac{\left( \omega _{3}x+\omega _{2}h_{2}+\omega _{1}\right) \pi \left( n-%
\frac{1}{2}\right) }{\omega _{3}+\omega _{2}h_{2}+\omega _{1}}\right]
+O\left( \frac{1}{n}\right) ,\text{\ }x\in \left( h_{2},1\right] .%
\end{array}%
\right.
\end{equation*}%
Case 4: $\beta _{2}^{\prime }=0$, $\sin \alpha =0$%
\begin{equation*}
\phi _{\lambda _{n}}\left( x\right) =\left\{
\begin{array}{c}
-\frac{\omega _{1}+\omega _{2}}{\omega _{1}}\frac{\cos \alpha }{\pi n}\sin %
\left[ \frac{\omega _{1}\pi n\left( x+1\right) }{\omega _{2}+\omega _{1}}%
\right] +O\left( \frac{1}{n^{2}}\right) ,\text{\ }x\in \left[
-1,h_{1}\right) , \\
\frac{-\gamma _{1}}{\delta _{1}}\frac{\omega _{1}+\omega _{2}}{\omega _{1}}%
\frac{\cos \alpha }{\pi n}\sin \left[ \frac{\left( \omega _{2}x+\omega
_{1}h_{1}+\omega _{1}\right) \pi n}{\omega _{2}+\omega _{1}h_{1}+\omega _{1}}%
\right] +O\left( \frac{1}{n^{2}}\right) ,\text{ }x\in \left(
h_{1},h_{2}\right) , \\
\frac{-\gamma _{1}\gamma _{3}}{\delta _{1}\delta _{3}}\frac{\omega
_{1}+\omega _{2}}{\omega _{1}}\frac{\cos \alpha }{\pi n}\sin \left[ \frac{%
\left( \omega _{3}x+\omega _{2}h_{2}+\omega _{1}\right) \pi n}{\omega
_{3}+\omega _{2}h_{2}+\omega _{1}}\right] +O\left( \frac{1}{n^{2}}\right) ,%
\text{ }x\in \left( h_{2},1\right] .%
\end{array}%
\right.
\end{equation*}%
All these asymptotic formulas hold uniformly for $x.$
\end{theorem}

\begin{proof}
Let us consider only the Case 1. Inserting (3.22) in the integral term of
(3.27), we easily see that%
\begin{equation*}
\int_{h_{2}}^{x}\sin \left[ \mu \omega _{3}\left( x-y\right) \right] q\left(
y\right) \phi _{3}\left( y,\lambda \right) dy=O\left( \exp \left( \left\vert
t\right\vert \left( \omega _{3}x+\omega _{2}h_{2}+\omega _{1}\right) \right)
\right) .
\end{equation*}%
Inserting in (3.20) yields%
\begin{eqnarray}
\phi _{3}\left( x,\lambda \right) &=&\frac{\gamma _{1}\gamma _{3}}{\delta
_{1}\delta _{3}}\sin \alpha \cos \left[ \mu \left( \omega _{3}x+\omega
_{2}h_{2}+\omega _{1}\right) \right]  \notag \\
&&+O\left( \frac{1}{\left\vert \mu \right\vert }\exp \left\vert t\right\vert
\left( \omega _{3}x+\omega _{2}h_{2}+\omega _{1}\right) \right) .
\TCItag{3.35}
\end{eqnarray}%
We already know that all eigenvalues are real. Furthermore, putting $\lambda
=-H,$ $H>0$ in (3.33) we infer that $\omega \left( -H\right) \rightarrow
\infty $ as $H\rightarrow +\infty $, and so $\omega \left( -H\right) \neq 0$
for sufficiently large\ $R>0$. Consequently, the set of eigenvalues is
bounded below. Letting $\sqrt{\lambda _{n}}=\mu _{n}$ in (3.35) we now obtain%
\begin{equation*}
\phi _{3}\left( x,\lambda _{n}\right) =\frac{\gamma _{1}\gamma _{3}}{\delta
_{1}\delta _{3}}\sin \alpha \cos \left[ \mu _{n}\left( \omega _{3}x+\omega
_{2}h_{2}+\omega _{1}\right) \right] +O\left( \frac{1}{\mu _{n}}\right)
\end{equation*}%
since $t_{n}=lm\mu _{n}$ for sufficiently large $n$. After some calculation,
we easily see that%
\begin{equation*}
\cos \left[ \mu _{n}\left( \omega _{3}x+\omega _{2}h_{2}+\omega _{1}\right) %
\right] =\cos \left[ \frac{\left( \omega _{3}x+\omega _{2}h_{2}+\omega
_{1}\right) \pi \left( n-1\right) }{\omega _{3}+\omega _{2}h_{2}+\omega _{1}}%
\right] +O\left( \frac{1}{n}\right) .
\end{equation*}%
Consequently,%
\begin{equation*}
\phi _{3}\left( x,\lambda _{n}\right) =\frac{\gamma _{1}\gamma _{3}}{\delta
_{1}\delta _{3}}\sin \alpha \cos \left[ \frac{\left( \omega _{3}x+\omega
_{2}h_{2}+\omega _{1}\right) \pi \left( n-1\right) }{\omega _{3}+\omega
_{2}h_{2}+\omega _{1}}\right] +O\left( \frac{1}{n}\right) .
\end{equation*}%
In a similar method, we can deduce that%
\begin{equation*}
\phi _{2}\left( x,\lambda _{n}\right) =\frac{\gamma _{1}}{\delta _{1}}\sin
\alpha \cos \left[ \frac{\left( \omega _{2}x+\omega _{1}h_{1}+\omega
_{1}\right) \pi \left( n-1\right) }{\omega _{2}+\omega _{1}h_{1}+\omega _{1}}%
\right] +O\left( \frac{1}{n}\right) ,
\end{equation*}%
and%
\begin{equation*}
\phi _{1}\left( x,\lambda _{n}\right) =\sin \alpha \cos \left[ \frac{\omega
_{1}\pi \left( n-1\right) \left( x+1\right) }{\omega _{2}+\omega _{1}}\right]
+O\left( \frac{1}{n}\right) .
\end{equation*}%
Thus the proof of the theorem completed in Case 1. The proofs for the other
cases are similar.
\end{proof}




\begin{thebibliography}{99}
\bibitem{Svt} M. Demirci , Z. Akdo\u{g}an, O. Sh. Mukhtarov, Asymptotic
behavior of eigenvalues and eigenfuctions of one discontinuous
boundary-value problem, Intenational Jounal of Computational Cognition 2(3)
(2004) 101--113.

\bibitem{} M. Kadakal, O. Sh. Mukhtarov, F.\c{S}. Muhtarov, Some spectral
problems of Sturm-Liouville problem with transmission conditions, Iranian
Journal of Science and Technology, 49(A2) (2005) 229-245.

\bibitem{} Z. Akdo\u{g}an, M. Demirci, O. Sh. Mukhtarov, Green function of
discontinuous boundary-value problem with transmission conditions.
Mathematical Models and Methods in Applied Sciences 30 (2007) 1719-1738.

\bibitem{} E. Tun\c{c}, O. Sh. Mukhtarov, Fundamental solutions and
eigenvalues of one boundary-value problem with transmission conditions,
Appl. Math. Comput. 157 (2004) 347--355.

\bibitem{} O. Sh. Mukhtarov, E. Tun\c{c}, Eigenvalue problems for
Sturm--Liouville equations with transmission conditions, Israel J. Math. 144
(2004) 367--380.

\bibitem{} O. Sh. Mukhtarov, M. Kadakal, F.\c{S}. Muhtarov, Eigenvalues and
normalized eigenfunctions of discontinuous Sturm--Liouville problem with
transmission conditions, Rep. Math. Phys. 54 (2004) 41--56.

\bibitem{} O. Sh. Mukhtarov, M. Kandemir, N. Kuruoglu, Distribution of
eigenvalues for the discontinuous boundary value problem with functional
manypoint conditions, Israel J. Math. 129 (2002) 143--156.

\bibitem{} O. Sh. Mukhtarov, S. Yakubov, Problems for ordinary differential
equations with transmission conditions, Appl. Anal. 81(5) (2002) 1033--1064.

\bibitem{} O. Sh. Mukhtarov, Discontinuous Boundary value problem with
spectral parameter in boundary conditions, Turkish J. Math., 18 (2) (1994)
183-192.

\bibitem{} O. Sh. Mukhtarov and M. Kadakal, On a Sturm-Liouville type
problem with discontinuous in two-points, Far East Journal of Applied
Mathematics, 19 (3) (2005) 337--352.

\bibitem{} M. Kadakal, O. Sh. Mukhtarov, Sturm-Liouville problems with
discontinuities at two points, Comput. Math. Appl.,54 (2007) 1367-1379.

\bibitem{} O. Sh. Mukhtarov and M. Kadakal, Some spectral properties of one
Sturm-Liouville type problem with discontinuous weight, \textit{Siberian
Mathematical Journal}, 46 (2005) 681-694.

\bibitem{} R. P. Gilbert and H. C. Howard, On the singularities of
Sturm-Liouville expansions: II, \textit{Applicable Analysis}, 2 (1972)
269-282.

\bibitem{} E. C. Titchmarsh, Eigenfunctions Expansion Associated with Second
Order Differential Equations. I, Oxford Univ. Press, London, 1962.

\bibitem{} J. Walter, Regular eigenvalue problems with eigenvalue parameter
in the boundary conditions, \textit{Math. Z.} 133 (1973) 301--312.

\bibitem{} D. B. Hinton, An expansion theorem for an eigenvalue problem with
eigenvalue parameter in the boundary condition, \textit{Quarterly Journal of
Mathematics} 30 (1979) 33--42.

\bibitem{} A. A. Shkalikov, Boundary value problems for ordinary
differential equations with a parameter in boundary condition, \textit{Trudy
Sem. Imeny I. G. Petrowsgo} 9 (1983) 190--229.

\bibitem{fu} C. T. Fulton, Two-point boundary value problems with eigenvalue
parameter contained in the boundary conditions, \textit{Proc. Roy. Soc.
Edinburgh Sect. }A 77 (1977) 293--308.

\bibitem{} S. Yakubov and Ya. Yakubov, Abel basis of root functions of
regular boundary value problems, \textit{Math. Nachr.} 197 (1999) 157--187.

\bibitem{} S. Yakubov and Ya. Yakubov, Differential-Operator Equations.
Ordinary and Partial Differential Equations, Chapman and Hall/CRC, Boca
Raton, 2000.

\bibitem{} E. \c{S}en and A. Bayramov, Calculation of eigenvalues and
eigenfunctions of a discontinuous boundary value problem with retarded
argument which contains a spectral parameter in the boundary condition,
\textit{Mathematical and Computer Modelling} 54 (2011) 3090-3097.

\bibitem{} A. N. Tikhonov and A. A. Samarskii, Equations of Mathematical
Physics, Pergamon, Oxford; New York,1963.

\bibitem{} P. A. Binding, P. J. Browne and B. A. Watson, Sturm-Liouville
problems with boundary conditions rationally dependent on the eigenparameter
II. \textit{J. Comput. Appl. Math.} 148 (2002) 147-169.

\bibitem{} Kh. R. Mamedov, On an inverse scattering problem for a
discontinuous Sturm-Liouville equation with a spectral parameter in the
boundary condition, \textit{Boundary Value Problems} 2010:171967 (2010) 17
pages.

\bibitem{} M. Ront\'{o} and A. M. Samoilenko, Numerical-Analytic Methods in
Theory of Boundary-Value Problems, World Scientific, Singapore (2000), 455 p.
\end{thebibliography}
\end{document}